\PassOptionsToPackage{dvipsnames}{xcolor}
\documentclass[11pt]{article}

\usepackage{amsmath,amsthm,amssymb}

\usepackage{fullpage}
\usepackage{titlesec}
\usepackage{enumerate}

\usepackage{xcolor}

\usepackage{tikz}
\usetikzlibrary{calc,positioning}

\usepackage{todonotes}

\usepackage{relsize}
\usepackage{comment}
\usepackage{csquotes}

\usepackage[unicode,hidelinks]{hyperref}
\usepackage{xurl} 
\usepackage[nameinlink,capitalize,noabbrev]{cleveref}
\usepackage{placeins}

\newcommand{\doi}[1]{\href{https://doi.org/#1}{doi:\,#1}}
\newcommand{\arxiv}[1]{\href{https://arxiv.org/abs/#1}{arXiv:\,#1}}

\newtheorem{theorem}{Theorem}[section]

\newtheorem{definition}[theorem]{Definition}

\newtheorem{conjecture}[theorem]{Conjecture}

\newcommand{\eps}{\varepsilon}

\newcommand{\R}{\mathbb{R}}
\newcommand{\Z}{\mathbb{Z}}
\newcommand{\N}{\mathbb{N}}

\DeclareMathAccent{\widehat}{\mathord}{largesymbols}{"62}

\title{$2$-large sets are sets of Bohr recurrence}
\author{
Ryan Alweiss
\thanks{Department of Pure Mathematics and Mathematical Statistics and Trinity College, University of Cambridge.
Email: {\tt ra699@cam.ac.uk}.
Research supported by an NSF Mathematical Sciences Postdoctoral Fellowship.}\\
}

\begin{document}
\maketitle

\begin{abstract}
	Let $\alpha_1,\ldots,\alpha_d$ be real numbers, fix $\delta>0$,
and let $S$ be the set of positive integers $s$ such that
$\|\alpha_i s\|_{\R/\Z}>\delta$ for some $1\le i\le d$.
We prove that $S$ is not \enquote{$2$-large}, i.e. there is
a $2$-coloring of $\N$ that avoids arbitrarily long
monochromatic arithmetic progressions with common
differences in $S$. The case $d=1$ was previously known; here we prove the
result in every finite dimension.
\end{abstract}

\section{Introduction}

One of the oldest and most classic results in Ramsey theory is van der Waerden's theorem, that any finite coloring of $\N$ has arbitrarily long arithmetic progressions in some color \cite{vdw}.

A natural question, with deep connections to dynamics, is what happens if the common differences are restricted to be in some set $S$. In the 1990s, Brown, Graham, and Landman \cite{bgl} defined a set $S$ to be \enquote{large} if van der Waerden holds with restricted differences in the set $S$.

\begin{definition}[Large Sets \cite{bgl}]

We say that $S \subset \N$ is \enquote{large} if whenever $\N$ is finitely colored, there are arbitrarily long monochromatic arithmetic progressions with common differences in $S$.
	
\end{definition}

A large set is also referred to as a \emph{set of topological multiple recurrence} in the dynamics literature.  See \cite{alweiss} for a discussion of the equivalence.

Brown, Graham, and Landman also defined \enquote{2-large}, an analogous notion for $2$-colorings.

\begin{definition}[2-Large Sets \cite{bgl}]

We say that $S \subset \N$ is \enquote{2-large} if whenever $\N$ is two-colored, there are arbitrarily long monochromatic arithmetic progressions with common differences in $S$.
	
\end{definition}

It is trivial to see that large sets are $2$-large sets.  Brown, Graham, and Landman conjectured that the reverse is also true.  This is referred to as the \enquote{Large Sets Conjecture} in the literature.

\begin{conjecture}[Large Sets Conjecture \cite{bgl}] If $S$ is $2$-large, then $S$ is large.
\end{conjecture}

While their paper was published in 1999, the Open Problem Garden claims \cite{open} that the conjecture was formulated in 1995.  Shortly after the original formulation of the large sets conjecture, Bergelson and Leibman \cite{bl} proved that if $S=P(\N)$, where $P\in\Z[x]$ is nonconstant,
$P(0)=0$, and $P(n)>0$ for every $n\in\N$, then $S$ is large. This is a special case of the polynomial van der Waerden theorem of Bergelson and Leibman.

Perhaps it is easier to think about which sets are \emph{not} large.  We must determine if all such sets are not $2$-large.  But what do non-large sets look like?  This is in fact the main subject of the paper of Brown, Graham, and Landman \cite{bgl}.  Until recently the only known examples of non-large sets were sets that were not sets of \enquote{nil-Bohr recurrence}, which are known to be non-large, and \cite{hsy} asked whether or not these are the only such sets.  The author gave the first new example of a non-large set \cite{alweiss}. See also the expository account of \cite{green}.

Let us start with a basic example of a non-large set.  If $S$ consists only of numbers $s$ such that
$\|s\alpha\|_{\R/\Z}>\eps$, for some fixed $\alpha\in\R$
and $\eps>0$, then $S$ is not large, because we can color
$x$ by the interval containing $\{\alpha x\}$ in a
partition of $[0,1)$ into intervals of length less than
$\eps$. If $x,x+d$ are the same color, then
$\|d\alpha\|_{\R/\Z}<\eps$, so $d\notin S$.  It is known (see \cite{sohail} and Lemma 7.4 of \cite{kra}) that such sets are not $2$-large.

\begin{definition}[One-dimensional Bohr recurrence] A set $S$ is a set of \emph{one-dimensional Bohr recurrence} if for any $\alpha \in \R$ and $\eps>0$, there exists $s \in S$ so that $||s \alpha||_{\R/\Z}<\eps$.\end{definition}

In other words, it is known that if $S$ is $2$-large, then it is a set of one-dimensional Bohr recurrence.

\begin{definition}[Bohr recurrence] A set $S$ is a set of \emph{Bohr recurrence} if for any integer $d \ge 1$, any $\alpha_1, \cdots, \alpha_d \in \R$, any $\eps>0$, there exists $s \in S$ so that $||s \alpha_i||_{\R/\Z}<\eps$ for $1 \le i \le d$.
\end{definition}

If a set $S$ is \emph{not} a set of Bohr recurrence, then it is readily seen not to be large.  One can color $n$ by taking the map $n \to ( \{\alpha_1 n \}, \cdots, \{ \alpha_d n \}) \in (\R/\Z)^d$, partitioning $(\R/\Z)^d$ into half-open cubes of side length less than $\eps$, and coloring $n$ by which cube its image lands in.  More generally, if $S$ is not a set of \enquote{nil-Bohr recurrence}, then it is not large.  We are ready to state our main result, which establishes the finite-union extension proposed
in \cite[Section~3]{sohail}.

\begin{theorem}[Main Theorem]\label{thm:main}
	Let $S\subseteq\N$. If there exist an integer $d\ge1$,
$\alpha_1,\ldots,\alpha_d\in\R$, and $\eps>0$ such that
for all $s\in S$ there exists $1\le i\le d$ with $||s \alpha_i||_{\R/\Z}>\eps$ (i.e. $S$ is not a set of Bohr recurrence), then $S$ is not $2$-large.
\end{theorem}

We also state the contrapositive, because it is a bit cleaner.

\begin{theorem}[Main Theorem, contrapositive]\label{thm:clean}
	If $S$ is $2$-large, then it is a set of Bohr recurrence.
\end{theorem}

This result generalizes all the examples of non-large sets from \cite{bgl}.  In particular, \cite{bgl} presents many examples of non-large sets which are in fact not Bohr recurrent, such as lacunary sets.

We will not formally define \enquote{nil-Bohr recurrence} in the paper.
It is the higher-degree generalization of Bohr recurrence, so for
instance if $\|\alpha s^2\|_{\R/\Z}>\delta$ for all $s\in S$, for some
fixed $\alpha\in\R$ and $\delta>0$, then the identity
$(x+2s)^2-2(x+s)^2+x^2=2s^2$ gives an obstruction to largeness:
color $x$ according to the interval containing $\alpha x^2/2$
modulo $1$ in a sufficiently fine partition of $[0,1)$.
One can replace $s^2$ with a suitable higher-degree polynomial or
\enquote{generalized polynomial} involving floor functions.
See \cite{alweiss} for details.

\section{Coloring and Proof Outline}

Before describing the proof of~\Cref{thm:main}, we give a remark about the technique from \cite{sohail} and \cite{kra} and an obstruction to generalizing it directly, and we explain what we need to do instead.  The proof of \cite{sohail} and \cite{kra} proceeds by considering a map $f$ from $\Z$ to $\R/\Z$ taking $n \to \{n \alpha\}$, and coloring $\R/\Z$ in two colors.  In particular, if the interval $[0,1/2)$ is red and the interval $[1/2,1)$ blue then the length of a monochromatic arithmetic progression whose common difference is bounded away from $0$ in $\R/\Z$ will be bounded.  The following observation gives an obstruction to
straightforwardly extending this sort of argument.

\begin{theorem}
Let $d$ be a sufficiently large positive integer.
Then every two-coloring of $(\R/\Z)^d$ contains
a monochromatic three-term arithmetic progression
$a,a+s,a+2s$, where
$s\in\{0,1/3\}^d\setminus\{0\}^d$.
Consequently, the sequence $(a+ns)_{n\ge0}$ is monochromatic.
\end{theorem}

\begin{proof}
Consider only the points $\{0,1/3,2/3\}^d$.
By Hales--Jewett~\cite{hj} we can find a monochromatic combinatorial
line, which corresponds to the required three-term
arithmetic progression. Since $3s=0$, the sequence
$(a+ns)_{n\ge0}$ repeats these three points.
\end{proof}

Instead we need to take the universal cover of the torus $(\R/\Z)^d$ and consider the map $f: \Z \to \R^d$ given by $x \to (\alpha_1 x, \cdots, \alpha_d x)$.
Then we color $\mathbb{R}^d$ in the following way.  Pick a sufficiently large integer $k$ depending on $d$
and $\delta$, and then pick $N_0$ to be a sufficiently
large factorial depending on $k$, $d$, and $\delta$.
Set $N_1=kN_0$.
We will have the coloring be periodic mod $N_1$ in all directions. 
So then it suffices to consider a coloring of $(\mathbb{R}/N_1 \Z)^d$, or $[0,N_1)^d$.  This object is naturally decomposed into $k^d$ translates of $[0,N_0)^d$, called cells. 
We will color each cell red or blue randomly.  We refer to this object, the $k^d$ cells all colored randomly, as the ``chessboard".
We pull this coloring back to $\R^d$ and then to $\Z$
through the maps
$\Z\to\R^d\to(\R/N_1\Z)^d$.
Then, if we take $N_2 \gg N_1$, any arithmetic progression with a common difference $x$ of length $N_2+1$ so that $||\alpha_i x||>\delta$ for some $i$ will meet both red and blue squares.  
We will show that there is a family of only $k^{O_d(1)}$
sets of cells, each containing $\Omega_d(k)$ cells,
such that every relevant progression visits every cell
in one of these sets. Each such set has probability
$2^{-\Omega_d(k)}$ of being monochromatic. 
So if $k$ is chosen large enough, a union bound gives us what we want by a simple probabilistic argument.  In some sense, this is a combinatorial phrasing of the fact that a torus rotation has zero entropy.

To show this, we must use that the arithmetic progression will return near its original position in the chessboard within $s=O_{d,1/\delta}(k^d)$ steps. 
If it does not go too close to its original position in $s$ steps, then going another $s$ steps, and another $s$ steps and so on will cause it to trace out a ``line" in far fewer than $N_2$ steps. 
If it is very close to the original position in $s$ steps, then we finish by a divisibility argument. 

In particular, we choose $N_1$ divisible by every possible
$s$. If the return displacement were less than $\delta$
in every coordinate, then each $\alpha_i x$ would be
within $\delta/s\le\delta$ of a multiple of $N_1/s$.
Since $N_1/s$ is an integer, this contradicts
$\|\alpha_i x\|_{\R/\Z}>\delta$ for some $i$.

\section{The Proof}

Let $\alpha_1, \cdots, \alpha_d \in \mathbb{R}$, and let $S$ be the set of positive integers $x$ so that $||\alpha_i x||>\delta$ for some $i$ and some fixed $\delta>0$.
We will prove $S$ is not $2$-large by exhibiting a $2$-coloring.

Define some sufficiently large integers $N_0 \gg k \gg C \gg \max(d, 1/\delta)$, so that $N_1=kN_0$ is divisible by all integers up to $(CN_1/N_0)^d=(Ck)^d$.  We can ensure this by for instance setting $N_0$ to be a sufficiently large factorial function. We will choose $N_2$ to be a sufficiently large multiple of $N_1$,
depending on $d$, $\delta$, and $N_1$, and then we will show that
there is a coloring so that every $(N_2+1)$-term arithmetic progression
with common difference in $S$ receives both colors.

Consider the map $f: \Z \to \mathbb{R}^d$ sending $x \to (\alpha_1x, \cdots, \alpha_d x)$ and the usual quotient map $g: \mathbb{R}^d \to (\mathbb{R}/N_1 \Z)^d$.  Composing these two maps, we get a map $F: \Z \to (\mathbb{R}/N_1 \Z)^d$. 
We will color the torus $(\mathbb{R}/N_1 \Z)^d$ red and blue, and by taking the pre-image under $F$ this gives a coloring of $\Z$. In order to color this torus, which has fundamental domain $[0,N_1)^d$ in $(\mathbb{R}/N_1 \Z)^d$, we decompose into $(N_1/N_0)^d$ copies of $[0,N_0)^d$ in the usual way.  
Each copy will be colored red or blue, uniformly and independently.
We claim that with high probability this coloring works, and any arithmetic progression of length $N_2+1$ with common difference in $S$ will have both red and blue elements.  
We call the object $(\mathbb{R}/N_1\Z)^d$ the ``$N_1$-torus" and we say it is decomposed into $(N_1/N_0)^d$ ``$N_0$-cells". 
Furthermore, we split each $N_0$-cell into $C^d$ ``mini-cells" in the obvious way.

Say we have an arithmetic progression in $\Z$ of the form $x, x+y, \cdots, x+N_2 y$ with $y \in S$.  We take its image under $F$ in the torus $(\mathbb{R}/N_1\Z)^d$. 
By pigeonhole, at least two of $F(x), F(x+y), \cdots, F(x+(CN_1/N_0)^dy)$ lie in the same mini-cell, say $F(x+ay)$ and $F(x+by)$.
This means for some $t=|a-b| \le (CN_1/N_0)^d \ll N_2$, we have that $F(x+ty)$ and $F(x)$ are close in the $N_1$-torus, differing by at most $N_0/C$ in each coordinate.

Since $t\le(Ck)^d$, our choices give $t\mid N_1\mid N_2$, so we may consider the sub-progression $x, x+ty, \cdots, x+N_2y$ of length $N_2/t + 1 \gg N_1$.
We first show that $F(ty)$ has some coordinate bounded away from $0$ in the $N_1$-torus, via a divisibility argument. 
Let $y' \in [0,N_1)^d$ be such that $F(y) \equiv y'$ in $(\mathbb{R}/N_1\Z)^d$. 
Then we have $ty' \in [0, tN_1)^d$, and by assumption $y'$ is at least $\delta$ away from every integer point in the $L^{\infty}$ distance. Say $ty'$ has distance $\eta$ from some point $(b_1N_1, \cdots, b_dN_1)$ where $b_i$ are integers, so that $F(ty)$ is $\eta$-close to $0$ in $(\mathbb{R}/N_1\Z)^d$ in $L^{\infty}$ distance. 
Then $y'$ has distance at most $\eta/t \le \eta$ from the integer point $(b_1\frac{N_1}{t}, \cdots, b_d\frac{N_1}{t})$, and so $\eta \ge \delta$.

Now, we have that $F(ty)$ has some coordinate at least
$\delta$-far from $0$ in the $N_1$-torus, and we can choose
a representative with all coordinates less than $N_0/C$
in magnitude. Since $N_2$ is sufficiently large, the
subprogression traces out a line segment of length at least
$(k+2)N_0$ in the covering space, where lengths are measured
in the $L^\infty$ norm.

Every such segment meets at least $k$ distinct cells modulo
$N_1$ in intervals of length at least $N_0/d$. One can see
this by projecting onto a coordinate of maximal speed:
within the interior of each complete layer of width $N_0$,
the crossings of the other cell boundaries split the line
into at most $d$ pieces, one of which has length at least
$N_0/d$.

\begin{figure}[!htbp]
  \centering
  \begin{tikzpicture}[x=1cm,y=1cm] 
    \pgfmathsetseed{20251209}              

    \foreach \x in {0,...,7} {
      \foreach \y in {0,...,7} {
        \pgfmathrandominteger{\coin}{0}{1} 
        \ifnum\coin=0
          \fill[red!40]  (\x,\y) rectangle ++(1,1);
        \else
          \fill[blue!40] (\x,\y) rectangle ++(1,1);
        \fi
      }
    }

    \draw[thick] (0,0) rectangle (8,8);

    \begin{scope}
      \clip (0,0) rectangle (8,8);

      \pgfmathsetmacro{\dx}{1.0}
      \pgfmathsetmacro{\dy}{1.3}
      \pgfmathsetmacro{\norm}{sqrt(\dx*\dx + \dy*\dy)}
      \pgfmathsetmacro{\stepsize}{0.37 / \norm}

      \pgfmathsetmacro{\xstart}{0.3}
      \pgfmathsetmacro{\ystart}{0.7}

      \pgfmathsetmacro{\T}{40}
      \draw[line width=0.03cm]
        ({\xstart - \T*\dx},{\ystart - \T*\dy}) --
        ({\xstart + \T*\dx},{\ystart + \T*\dy});

      \foreach \k in {-40,...,40} {
        \pgfmathsetmacro{\t}{\k * \stepsize}
        \pgfmathsetmacro{\xx}{\xstart + \t*\dx}
        \pgfmathsetmacro{\yy}{\ystart + \t*\dy}
        \fill[black] (\xx,\yy) circle[radius=0.10cm];
      }
    \end{scope}
  \end{tikzpicture}
 \caption{A subprogression crossing the colored cells.
  On the $N_1$-torus, this line wraps around the board.}
\end{figure}
\FloatBarrier
In fact, there are only $k^{O_d(1)}$ choices of sets of such
cells (which we call an \emph{orbit}) that we need to consider.
To see this, take its initial segment of length $(k+2)N_0$,
choosing the lift of its starting point in $[0,N_1)^d$.
Mark this segment at spacing $N_0/(8d)$ and retain the
lifted cells containing three consecutive marks,
then reduce these cells modulo $N_1$. The preceding observation
gives at least $k$ retained cells. In each coordinate, the
cell assignments of the $O_d(k)$ marks change across only
$O_d(k^2)$ lines in the two parameters describing the starting
point and direction; their arrangement has $O_d(k^4)$ faces,
including boundary faces. Taking the product over coordinates
gives the claimed count. Taking $C\ge8d$ ensures that the
arithmetic progression meets every retained cell.

Hence, there are only $k^{O_d(1)}$ possible orbits, and in
a random coloring each one has probability at most $2^{1-k}$
of being monochromatic. So if $k$ is chosen large enough,
a union bound gives a coloring in which every $(N_2+1)$-term
arithmetic progression with common difference in $S$
receives both colors.

\section{Further Directions}

In general, a set is not \enquote{large} in the sense of \cite{bgl} if it is not a set of (topological) \emph{multiple recurrence} for some minimal dynamical system.  Here we have handled torus rotations, which are a special kind of dynamical system.  One could ask also about increasingly more general classes of dynamical systems such as nilsystems, skew products, and general distal systems.  A potential approach to settling the $2$-large versus large conjecture in the positive is to settle it for some classes of dynamical systems and then to use an appropriate structure theorem to extend the result to all minimal topological dynamical systems.  A similar approach to Katznelson's Question is outlined in \cite{gkr}.

We know how to settle the $2$-large versus large conjecture for some nilrotations.  If $\alpha$ is irrational, then for any fixed $\delta>0$ the
set $\{n\in\N:\|n^2\alpha\|_{\R/\Z}>\delta\}$ can be shown
to be not $2$-large by a similar argument, using standard
facts about polynomial equidistribution
(see \cite{ET48,Weyl16}). In subsequent joint work with Farhangi~\cite{af}, we build
on the ideas of this paper to prove that every $2$-large
set is a set of nil-Bohr recurrence. The construction of \cite{alweiss} can be rephrased as coming from a skew product on $(\R/\Z)^5$.  As noted in \cite{af}, the example from \cite{alweiss} is also not $2$-large, though it remains open whether Green's modification \cite{green} is $2$-large.  However, we do not have strong opinions on what happens for all dynamical systems, i.e. on whether the $2$-large versus large conjecture of \cite{bgl} is true.

\section{Acknowledgments}

The author thanks Sohail Farhangi for helpful comments. He thanks GPT-5 for LaTeX help, and thanks GPT-6 for several suggestions that improved the exposition of the paper.


\begin{thebibliography}{99}

\bibitem{alweiss}
R.~Alweiss,
\emph{New Obstacles to Multiple Recurrence},
Preprint (2025).
\arxiv{2511.21680}.

\bibitem{af}
R.~Alweiss and S.~Farhangi,
\emph{$2$-large sets are sets of nil-Bohr recurrence},
Preprint (2026).
\arxiv{2610.16605}.

\bibitem{bl}
V.~Bergelson and A.~Leibman,
\emph{Polynomial extensions of van der Waerden's and Szemer\'edi's theorems},
J. Amer. Math. Soc. \textbf{9} (1996), no.~3, 725--753.
\doi{10.1090/S0894-0347-96-00194-4}.
\url{https://people.math.osu.edu/bergelson.1/PolSz.pdf}.

\bibitem{bgl}
T.~C.~Brown, R.~L.~Graham and B.~M.~Landman,
\emph{On the set of common differences in van der Waerden's theorem on arithmetic progressions},
Canad. Math. Bull. \textbf{42} (1999), no.~1, 25--36.
\doi{10.4153/CMB-1999-003-9}.

\bibitem{open}
T.~C.~Brown, R.~L.~Graham and B.~M.~Landman,
\emph{The large sets conjecture},
Open Problem Garden.
\url{https://www.openproblemgarden.org/op/the_large_sets_conjecture}.

\bibitem{ET48}
P.~Erd\H{o}s and P.~Tur\'an,
\emph{On a problem in the theory of uniform distribution. I, II},
Proc. Kon. Ned. Akad. Wetensch. \textbf{51} (1948), 1146--1154, 1262--1269
(= Indag. Math. \textbf{10} (1948), 370--378, 406--413).
Part~I: \url{https://users.renyi.hu/~p_erdos/1948-02.pdf}.
Part~II: \url{https://users.renyi.hu/~p_erdos/1948-03.pdf}.

\bibitem{sohail}
S.~Farhangi and J.~Grytczuk,
\emph{Distance graphs and arithmetic progressions},
Integers \textbf{21A} (2021), Paper No.~A11, 6~pp.
\doi{10.5281/zenodo.10687810}.
\url{https://math.colgate.edu/~integers/graham11/graham11.pdf}.

\bibitem{gkr}
D.~Glasscock, A.~Koutsogiannis and F.~K.~Richter,
\emph{On Katznelson's question for skew-product systems},
Bull. Amer. Math. Soc. (N.S.) \textbf{59} (2022), no.~4, 569--606.
\doi{10.1090/bull/1764}.

\bibitem{green}
B.~Green,
\emph{On Alweiss's example for multiple recurrence},
Preprint (2026).
\arxiv{2607.24594}.

\bibitem{hj}
A.~W.~Hales and R.~I.~Jewett,
\emph{Regularity and positional games},
Trans. Amer. Math. Soc. \textbf{106} (1963), no.~2, 222--229.
\doi{10.2307/1993764}.

\bibitem{kra}
B.~Host, B.~Kra and A.~Maass,
\emph{Variations on topological recurrence},
Monatsh. Math. \textbf{179} (2016), no.~1, 57--89.
\doi{10.1007/s00605-015-0765-0}.
\arxiv{1408.2728}.
\url{https://sites.math.northwestern.edu/~kra/papers/top-rec.pdf}.

\bibitem{hsy}
W.~Huang, S.~Shao and X.~Ye,
\emph{Nil Bohr-sets and almost automorphy of higher order},
Mem. Amer. Math. Soc. \textbf{241} (2016), no.~1143, v+83~pp.
\doi{10.1090/memo/1143}.

\bibitem{vdw}
B.~L.~van der Waerden,
\emph{Beweis einer Baudetschen Vermutung},
Nieuw Arch. Wisk. (2) \textbf{15} (1927), 212--216.

\bibitem{Weyl16}
H.~Weyl,
\emph{\"Uber die Gleichverteilung von Zahlen mod. Eins},
Math. Ann. \textbf{77} (1916), 313--352.
\doi{10.1007/BF01475864}.

\end{thebibliography}
\end{document}